\documentclass[10pt, letterpaper,  conference]{ieeeconf}

\IEEEoverridecommandlockouts                              

\usepackage{graphicx}
\usepackage{amsmath}
\usepackage{amssymb}

\usepackage{dsfont}
\usepackage{tikz}

\usepackage{amsthm}
\newtheorem{theorem}{Theorem}

\newtheorem{lemma}[theorem]{Lemma}

\newtheorem{corollary}[theorem]{Corollary}

\newtheorem{definition}{Definition}
\newtheorem{example}{Example}
\newtheorem{remark}{Remark}

\usepackage{xcolor}

\definecolor{cof}{RGB}{219,144,71}
\definecolor{pur}{RGB}{186,146,162}
\definecolor{greeo}{RGB}{91,173,69}
\definecolor{greet}{RGB}{52,111,72}

\newcommand{\bal}[1] {\ensuremath{\left(\begin{array}{#1}}}
\newcommand{\ear} {\ensuremath{\end{array}\right)}}

\newcommand{\bals}[1] {\ensuremath{\left[\begin{array}{#1}}} 
\newcommand{\ears} {\ensuremath{\end{array} \right] }} 

\DeclareMathOperator{\spa}{span}

\DeclareMathOperator{\diag}{diag}

\DeclareMathOperator{\vect}{vec}

\newcommand{\one} {\ensuremath{\mathds{1} }} 

\let\leq\leqslant
\let\geq\geqslant

\newcommand{\calA}{\ensuremath{\mathcal{A}}}
\newcommand{\calB}{\ensuremath{\mathcal{B}}}

\newcommand{\calI}{\ensuremath{\mathcal{I}}}

\newcommand{\calT}{\ensuremath{\mathcal{T}}}

\newcommand{\bmat}{\begin{matrix}}
\newcommand{\emat}{\end{matrix}}
\newcommand{\bbm}{\begin{bmatrix}}
\newcommand{\ebm}{\end{bmatrix}}
\newcommand{\bpm}{\begin{pmatrix}}
\newcommand{\epm}{\end{pmatrix}}
\newcommand{\bse}{\begin{subequations}}
\newcommand{\ese}{\end{subequations}}
\newcommand{\beq}{\begin{equation}}
\newcommand{\eeq}{\end{equation}}
\newcommand{\ben}{\begin{enumerate}}
\newcommand{\een}{\end{enumerate}}
\newcommand{\beni}{\renewcommand{\labelenumi}{\roman{enumi}.}
\renewcommand{\theenumi}{\roman{enumi}}\begin{enumerate}}
\newcommand{\eeni}{\end{enumerate}\renewcommand{\labelenumi}{\arabic{enumi}.}
\renewcommand{\theenumi}{\arabic{enumi}}}
\newcommand{\bena}{\renewcommand{\labelenumi}{\alpha{enumi}.}
\renewcommand{\theenumi}{\alpha{enumi}}\begin{enumerate}}
\newcommand{\eena}{\end{enumerate}\renewcommand{\labelenumi}{\arabic{enumi}.}
\renewcommand{\theenumi}{\arabic{enumi}}}
\newcommand{\bit}{\begin{itemize}}
\newcommand{\eit}{\end{itemize}}

\newcommand{\R}{\ensuremath{\mathbb R}}

\newcommand{\N}{\ensuremath{\mathbb N}}

\newcommand{\E}{\ensuremath{\mathbb E}}
\newcommand{\Q}{\ensuremath{\mathbb Q}}

\title{\LARGE \bf
On the modeling of neural cognition for social network applications}

\author{Jieqiang Wei, Junfeng Wu, Marco Molinari, Vladimir Cvetkovic and Karl H. Johansson 
\thanks{*This work is supported by Knut and Alice Wallenberg Foundation, Swedish Research Council, and Swedish Foundation for Strategic Research.}
\thanks{J. Wei, J. Wu and K.H. Johansson are with the ACCESS Linnaeus Centre, School of Electrical Engineering. 
V. Cvetkovic is with School of Architecture and the built environment. 
M. Molinari is with Department of Energy Technology.                     
                      KTH Royal Institute of Technology,
                    SE-100 44 Stockholm, Sweden. Emails:
        {\tt\small \{jieqiang, junfengw, marcomo, vdc, kallej\}@kth.se}}
}

\begin{document}

\maketitle
\thispagestyle{plain}
\pagestyle{plain}

\begin{abstract}\label{s:Abstract}
In this paper, we study neural cognition in social network. A stochastic model is introduced and shown to incorporate two well-known models in Pavlovian conditioning and social networks as special case, namely Rescorla-Wagner model and Friedkin-Johnsen model. The interpretation and comparison of these model are discussed. We consider two cases when the disturbance is independent identical distributed for all time and when the distribution of the random variable evolves according to a markov chain. We show that the systems for both cases are mean square stable and the expectation of the states converges to consensus.  

\end{abstract}

\section{Introduction}\label{s:Introduction}

There is a broad interest in studying social networks including opinions in a population, e.g., \cite{Degroot_74,Xia2011}. Incorporating initial opinions and weighting them relative to the network impact was an important extension for understanding social influence networks and opinion change \cite{FriedkinJohnsen1999,friedkin2011social}; it was shown how initial opinions may persist whereby resulting in consensus or disagreement. Recent development considers potentially correlated opinions on several different issues and shows how this correlation may affect convergence \cite{Friedkin321,Tempo2015}. 

The social network analyses with focus on consensus building preconceive a cognitive process that deals with how subjects integrate conflicting influential opinions \cite{FriedkinJohnsen1999}. Traditionally, the cognitive and emotional processes are perceived as separate, originating from “affective” or “cognitive” brain regions. Recent evidence, however, supports a more integrated view of cognition and emotion, the physiological basis being the high connectivity areas of the barin (“hubs”) \cite{Pessoa08}. 
Hence, pure cognition for opinion building within social networks cannot be assumed a priori. 

Associative learning through conditioning is central to understanding many key aspects of emotions, such as fear  \cite{Field2006,Hermans2006,Kong2014, Mineka2002,Mineka2008}. The Rescorla-Wagner model \cite{RW1972} has been found very useful for describing animal as well as human conditioning in a variety of contexts \cite{Pearce1980,Siegel1996,Culver2015}. Recently, Epstein \cite{epstein2014agent_zero} used the Rescorla-Wagner model to study social behaviour. The central concept in the work of \cite{epstein2014agent_zero} is the notion of “dispostion” composed of an emotional part (by conditioning) and a rational part (by cognition); action follows when disposition is greater than a given threshold.   

The main contribution of this paper is to integrate the two approaches of social network opinion analysis following the Friedkin-Johnsen model \cite{FriedkinJohnsen1999} and conditioning using a generalised form of the Rescorla-Wagner model that accounts for interpersonal influences on emotions. We model the opinion as a dynamic variable in the sense of \cite{Tempo2015} for a single issue, however, we add conditioning to the dynamic process. The associative strength is loosely interpreted  as emotional disposition \cite{epstein2014agent_zero}, and is assumed to directly influence the opinion on a given issue. In other words, we assume proportionality between an opinion and emotional disposition. Generic properties of the integrated social network-conditioning model are presented, opening new possibilities for quantifying social dynamics that have yet to be explored. Particularly, it is shown that the mean square stability of our model depends on the matrices given by the proportionality.

The structure of the paper is as follows. In Section~\ref{s:Preliminaries}, we introduce some terminologies and notations. In Section~\ref{s:basic_model}, we present our model with comparison to existing works. The main analytical results are presented in Section~\ref{s:main} mainly with sufficient and necessary conditions for mean square stability of our model.
Then the simulation and conclusion follows in Sections ~\ref{s:simulation} and \ref{s:conclusion} respectively.


\section{Preliminaries}\label{s:Preliminaries}


Given a square matrix $A=(a_{ij})_{i,j=1}^n$, let $\rho(A)$ be its spectral radius. The matrix $A$ is \emph{Schur stable} if $\rho(A)<1$. The matrix is \emph{row stochastic} if $a_{ij}\geq 0$ and $\sum_{j=1}^n a_{ij}=1, \forall i.$ The terminologies about Markov chains are kept consistent with \cite{norris1998markov}.

Consider a measurable space $(\Omega, \calA)$ with a nonempty set $\Omega$ and a $\sigma-$algebra $\calA\subset 2^{\Omega}$. For any $A\in\calA$, we define the indicator function $\mathds{1}_{A}:\Omega\mapsto\{0,1\}$ as 
\begin{align}
\mathds{1}_{A}(\omega)=
\begin{cases}
1 & \textrm{ if } \omega\in A,\\
0 & \textrm{ otherwise }.
\end{cases}
\end{align}

For a linear stochastic difference equation 
\begin{align}\label{e:linear_sto}
x(k+1) = Hx(k)+\zeta(k), k=1,2,\ldots
\end{align}
with $x(k)\in\R^n$, initial condition $x(0)$, $H\in\R^{n\times n}$ a constant matrix and $\zeta(k)$ a random vector, we define the mean square stability of the system \eqref{e:linear_sto} as follows.

\begin{definition}\cite{Costa2005}
The linear system \eqref{e:linear_sto} is mean square stable if for any initial condition $x(0)$, there exists $\mu\in\R^n$ and matrix $\Q$ (independent of $x(0)$) such that 
\begin{enumerate}
\item $\lim_{k\rightarrow\infty}\|\E(x(k))-\mu\|=0$,
\item $\lim_{k\rightarrow\infty}\|\E(x(k)x(k)^\top)-\Q\|=0$.
\end{enumerate}
\end{definition}


With $\R_-,\R_+, \R_{\geq 0}$ and $\R_{\leqslant 0}$ we denote the sets of negative, positive, non-negative, non-positive  real numbers, respectively. A positive definite, positive semi-definite matrix $M$ is denoted as $M>0$, $M\geq 0$, respectively. 
The symbol $\one_n$ represents a $n$-dimensional column vector with each entry being $1$.  We will drop the subscript $n$ when no confusion is possible. For a vector $\eta\in\R^n$, $\diag(\eta)$ is a diagonal matrix with the $i$th diagonal element equal to $\eta_i$. For a matrix $X=[x_1,x_2,\ldots,x_n]\in\R^{m\times n}$, $\vect(X)$ is the vectorization of $X$, i.e., $\vect(X)=[x^\top_1,x^\top_2,\ldots,x^\top_n]^\top.$ We use $\otimes$ to denote Kronecker product. For vectorization and Kronecker product, the following properties are frequently used in this paper: \textit{i)} $\vect(ABC)=(C^\top\otimes A)\vect(B)$; \textit{ii)} $(A\otimes B)(C\otimes D)=(AC)\otimes(BD)$, where $A,B,C$ and $D$ are matrices of compatible dimensions. For matrix $X$, its $1-$norm is the maximum absolute column sum, i.e., $\|X\|_1=\max_{1\leq j\leq n}\sum_{i=1}^m |X_{ij}|$.

\section{Models}\label{s:basic_model}

\subsection{Rescorla-Wagner model}

One of the most well-known models in \emph{Pavlovian} theory of associative learning, called \emph{Rescorla-Wagner model}, was proposed in \cite{RW1972}. In the classic Rescorla-Wagner model, the conditional stimulus has an associative value $x\in\R$, supposed to be proportional to the amplitude of the conditional response or to the proportion of conditional response triggered by the conditional stimulus. A typical Pavlovian conditioning session is a succession of several trials. Each trial is composed of the presentation of the conditional stimulus followed by the presentation of the unconditional stimulus. On each trial $k$, the associative value of the conditional stimulus are updated according to the following equation 
\begin{align*}
x(k+1) = x(k)+ \alpha \Big(r(k)- x(k) \Big), 
\end{align*}
where $r(k)\in\R$ is the intensity of the unconditional stimulus on that trial and $\alpha$ is a learning parameter, $x(k)\in\R$ is the associative strength between the conditional stimulus and the unconditional stimulus.

In more general Rescorla-Wagner models, more conditional stimulus can be incorporated. Each conditional stimulus has an associative value $x_i$, which is the associative strength of the $i$th conditional stimulus and the unconditional stimulus, namely some degree to which the conditional stimulus alone elicits the unconditional response. The associative value of all the conditional stimuluses are updated according to the following equation
\begin{align}\label{e:RW}
x_i(k+1) = x_i(k)+ \alpha \Big(r(k)-\sum_{i=1}^n x_i(k) \Big), i=1,\ldots,n,
\end{align}
where $n$ is the number of conditional stimulus on that trial.     
        
Rescorla-Wagner model is especially successful in explaining the block phenomenon in Pavlovian conditioning with experimental supports \cite{Siegel1996}.

\subsection{Epstein's model}

One of the major contributions of Epstein \cite{epstein2014agent_zero} is to establish the connection between a Pavlovian conditioning model, i.e., Rescorla-Wagner model, to neural cognition and human behavior in social networks. One methodology of describing human behavior is proposed by separating the human psychology into irrational, rational and social parts. The irrational component evolve according the Rescorla-Wagner model
\begin{equation}
x_i(k+1) = x_i(k)+\alpha(\lambda-x_i(k)),
\end{equation} 
where $\lambda$ is a random binary variable, which takes value one for emotion acquisition, and zero for emotion extinction. Then in the methodology proposed by \cite{epstein2014agent_zero}, human behavior (or action) depends on whether the summation of irrational and rational components of each person is larger than a given threshold. At each time step, everyone communicate with each others about the irrational part through a network.  One major "drawbacks" of this methodology is that the irrational part of each person can not be affected by the others dynamically. That motivates our generalization in this paper.

\subsection{Our model}

Consider a society composed by $n$ agents denoted $\calI=\{1,\ldots,n\}.$ We focus on the evolution of irrational component of the agents. We generalize Epstein's model by allowing the irrational component of each agent be affected by others dynamically, namely,
\begin{align}\label{e:our model}
x(k+1) = B x(k) + A (\lambda(k)-W x(k)),
\end{align}
where $x(k)\in\R^n$, $B$ and $W$ are row-stochastic matrices, $A$ is a diagonal matrix satisfying $0\leq A \leq I$, $\lambda$ is a random vector with each components being zero or one. The initial condition is set to be $x(0)=x_0\in\R^n$.

Note that the model \eqref{e:our model} contains \eqref{e:RW} as a special case by appropriate choice of matrices $A,B$ and $W$. Furthermore, by setting $A=I-\Lambda$, $B=W$ and $\lambda(k)$ being deterministic and identical for all $k$, our model corresponds to the Friedkin-Johnsen model \cite{friedkin2011social} given as
\begin{align}\label{e:FJ}
x(k+1) = \Lambda W x(k)+ (I-\Lambda)u, x(0)=u,
\end{align}
where $W$ is a row stochastic matrix, $\Lambda$ is a diagonal matrix satisfying $0\leq \Lambda \leq I$. This model is used to describe the dynamic of the (scalar) opinion of a community of $n$ social agents about one issue and $u$ is the initial opinion.

\section{Analytical properties}\label{s:main}

In this section, we study the stability of the state $x$ of the system \eqref{e:our model}. As the vector $\lambda$ introduce the randomness to $x$, one can expect the convergence of the system \eqref{e:our model} to depend on the distribution of $\lambda$. Here we consider two cases, namely  $\lambda(k)$ are independent and identically distributed (i.i.d.) for all $k$ or $\lambda$ obeys a Markov chain, respectively.

\subsection{I.i.d. $\lambda(k)$}\label{ss:iid}

In this subsection, we assume that the sequence of $\lambda(k)$ is i.i.d. . More precisely, $\forall k,$
\begin{equation}\label{e:iid}
\lambda_i(k) = \begin{cases}
0 & \textrm{ with probability } p_i,\\
1 & \textrm{ with probability } 1-p_i,
\end{cases}\quad \forall i\in\calI.
\end{equation}
Let us denote $p=[p_1,\ldots,p_n]^\top$. Then we have the expectation of $\lambda(k)=\mathds{1}-p.$ Moreover, suppose the components of $\lambda$ are correlated, and the covariance matrix of $\lambda(k)$ is $\Sigma$ which is positive definite.

In the first result, we characterize the probability of system \eqref{e:our model} having finite limit. In fact, we provide the sufficient and necessary condition for system \eqref{e:our model} converging to a finite limit almost surely.

\begin{theorem}
Consider the system \eqref{e:our model} with the random vector $\lambda$ satisfying \eqref{e:iid}, then the random variable $x(k)$ converge almost surely to a finite limit if and only if $B-AW$ is Schur stable. Furthermore, if $B-AW$ is Schur stable, the distribution of $\sum_{k=0}^\infty(B-AW)^kA\lambda(k)$ is the unique invariant distribution for the Markov chain $x(k)$.
\end{theorem}

\begin{proof}
First, by the condition \eqref{e:iid}, we have that 
\begin{align*}
\E \Big( \log^+\|\lambda(k)\| \Big)<\infty.
\end{align*}
Furthermore, it can be seen that $\log\|B-AW\|<0$ if and only if $\rho(B-AW)<1$, namely $B-AW$ is Schur stable. Then by using Theorem 2.1 in \cite{Diaconis1999}, the conclusion follows.
\end{proof}

The previous result presents the condition which guarantee the sequence \eqref{e:our model} has a finite limit. Next, we shall show the sufficient condition to guarantee the mean square stability of system \eqref{e:our model}.

\begin{theorem}\label{thm:main2}
Suppose the diagonal matrix $A$ satisfies $0<A<I$. Then the Markov chain $x(k)$ defined in \eqref{e:our model} is mean square stable if and only if $B-AW$ is Schur stable. Moreover, if $B-AW$ is Schur stable, the expectation of $x(k)$ converges to $(I-B+AW)^{-1}A(\mathds{1}-p)$.
\end{theorem}

\begin{proof}
For $k\geq 1,$ we write the dynamic of the expectation of $x(k)$ as 
\begin{align}
\E(x(k)) = & (B-AW)\E(x(k-1))+A\E(\lambda(k-1))\nonumber\\
= & (B-AW)\E(x(k-1))+A(\mathds{1}-p) \nonumber\\
= & (B-AW)^k x(0) \nonumber\\
  & + (1-p)\sum_{\ell=0}^{k-1} (B-AW)^\ell A\mathds{1}
\end{align}
where we have used the fact that $\E(\lambda(k))=\mathds{1}-p$ for all $k$. It can be verified that if $B-AW$ is Schur stable, we have 
\begin{align}
\lim_{k\rightarrow\infty}\E(x(k))= (I-B+AW)^{-1}A(\mathds{1}-p).
\end{align} 

Next, in order to show the mean square stability, we need to prove the stability of the expectation of the matrix series $\big(x(k)-\E(x(k))\big)\big(x(k)-\E(x(k))\big)^\top$. For the simplicity of the composition, we denote $x(k)-\E(x(k))$  as $S(k)$. Then we have
\begin{align}\label{e:MSS-proof}
 &\E(S(k)S^\top(k)) \nonumber\\
= & (B-AW)\E(S(k-1)S^\top(k-1))(B-AW)^\top \nonumber\\
& +  A \Sigma A^\top \nonumber\\
= & (B-AW)^k\E(S(0)S^\top(0))((B-AW)^k)^\top \nonumber\\
& + \sum_{\ell=0}^{k-1} (B-AW)^{\ell} A \Sigma A^\top ((B-AW)^\ell)^\top.
\end{align}
Notice that $\Sigma>0$, which implies that the summation in \eqref{e:MSS-proof} is converging to a finite matrix if and only if $\rho(B-AW)<1$. Hence  $\E(S(k)S^\top(k))$ converges to a finite matrix as $k\rightarrow\infty$ if and only if $\rho(B-AW)<1$. Hence the mean square stability of \eqref{e:our model} is proved.
\end{proof}

\begin{corollary}
If the system \eqref{e:our model} has Schur stable $B-AW$ and $p\in\spa\{\mathds{1}\}$, then the expectation of $x(k)$ converges to consensus. 
\end{corollary}

\begin{proof}
By Theorem \ref{thm:main2}, the expectation of $x(k)$ converges to $(I-B+AW)^{-1}A(\mathds{1}-p)$. Suppose $p=\alpha\mathds{1}$. We need to show that $(1-\alpha)(I-B+AW)^{-1}A\mathds{1}\in\spa\{\mathds{1}\}$. Indeed, we prove that $(1-\alpha)(I-B+AW)^{-1}A\mathds{1}=(1-\alpha)A\mathds{1}$ which can be seen by 
\begin{align}
(1-\alpha)A\mathds{1} & = (1-\alpha)(I-B+AW)\mathds{1}.
\end{align}
The previous equality holds for row stochastic matrix $B$ and $W$. Hence the expectation of $x(k)$ converge to $(1-\alpha)\mathds{1}$.
\end{proof}

In the previous result, one common assumption is the Schur stability of the matrix $B-AW$. For general row stochastic matrices $B$ and $W$ with $0<A<I$, it usually can not be guaranteed that the spectral radius of $B-AW$ is less than one. However for a special case, namely $B=W$, $\rho(B-AW)<1$ always holds. 

\begin{corollary}
When the stochastic matrices $B=W$ and the diagonal matrix $A$ satisfies $0<A<I$, then the $B-AW$ is Schur stable. 
\end{corollary}

\begin{proof}
Since $0<A<I$, the diagonal elements of $I-A$ belong to the open interval $(0,1)$. Then the sum of absolute value of the elements in each row of $B-AW$ is strictly less than one. Then by using Gershgorin disc theorem, we have the absolute value of all the eigenvalues of $B-AW$ are less than one. 
\end{proof}

\begin{remark}
If the model \eqref{e:our model} satisfy the $B=W$ and $0<A<I$, the expectation of $x(k)$ evolve according to Firedkin-Johnsen model \eqref{e:FJ} which is stable in this case, see Theorem~1 in \cite{Tempo2015}. 
\end{remark}

\begin{remark}\label{re:non-converging}
In the previous results, the Schur stability of $B-AW$ can guarantee the stability of the system \eqref{e:our model}. Generally speaking, the weakened conditions with $\rho(B-AW)\leq 1$ and $0\leq A\leq I$ can not guarantee that the expectation of $x(k)$ converges to a constant vector. A simulation for this case is given in Example \ref{ex: non-converging} in Section \ref{s:simulation}. 
\end{remark}

\subsection{Markovian $\lambda(k)$}\label{ss:markov}

In this subsection, we consider the case that the vector $\lambda(k)$, which takes value in the set $\mathbb{N}:=\{0,1\}^n\subset\R^n$, is a Markov chain with an initial distribution $\pi(0)\in\R^{2^n}$, i.e., $\lambda(0)=i$ with probability $\pi_{i}(0)$. Then for any $i\in\N$ the indicator function $\mathds{1}_{\{\lambda(k)=i\}}(w)=1$ if $\lambda(k)(w)=i$, and $0$ otherwise. The transition matrix for the Markov chain of $\lambda$ is $P$. Furthermore, we can rewrite the system \eqref{e:our model} as 
\begin{equation}\label{e:our_model_new}
x(k+1) = (B-AW) x(k) + D\lambda(k)a
\end{equation}
where $D\lambda(k)=\diag(\lambda(k))$ and the vector $a$ such that $A=\diag(a)$.



To simply the presentation, we introduce the following notations
\begin{align}\label{e:qiQi}
q_i(k) & = \E\big( x(k)\mathds{1}_{\{\lambda(k)=i\}} \big) \\
Q_i(k) & = \E\big( x(k)x^\top(k)\mathds{1}_{\{\lambda(k)=i\}} \big), 
\end{align}
and collect $q(k):=[\ldots,q^\top_i(k),\ldots]$ and $Q(k):=[\ldots,Q^\top_i(k),\ldots]$ for all $i\in\N$.
Then 
\begin{align}
\mu(k):=& \E(x(k)) = \sum_{i\in\N}q_i(k) \\
\Q(k):=& \E(x(k)x^\top(k)) = \sum_{i\in\N}Q_i(k)
\end{align}

\begin{lemma}
Suppose $B-AW$ is Schur stable and the initial distribution of $\lambda(0)$ is $\pi=\pi^*$ which is invariant for $P$, then the expectation $\mu(k)$ is converging to 
\begin{align}\label{e:markov_limit}
(I-B+AW)^{-1}\sum_{i\in\N} \pi_i \diag(i)a.
\end{align}
\end{lemma}

\begin{proof}
The dynamic of the expectation of $x(k)$ is given as 
\begin{align}
\mu(k)  = & (B-AW)^k \E(x(0)) \\
&+\sum_{\ell=0}^{k-1} \big((B-AW)^\ell \sum_{i\in\N} \pi_i \diag(i)a \big).
\end{align}

Then, if $B-AW$ is Schur stable, we have 
\begin{align}
\lim_{k\rightarrow \infty} \mu(k)=(I-B+AW)^{-1}\sum_{i\in\N} \pi_i \diag(i)a.
\end{align}
\end{proof}

Notice that by denoting 
\begin{align}
S:= \begin{bmatrix}
\vdots \\ \diag(i)a \\ \vdots
\end{bmatrix}, i\in\N,
\end{align}
we can rewrite $\sum_{i\in\N} \pi_i \diag(i)a= (\pi\otimes I_{n})S$.

\begin{theorem}
Suppose $P$ is irreducible and aperiodic. Then the sequence $x(k)$ is mean square stable if and only if $B-AW$ is Schur stable. Moreover, for any initial distribution $\pi(0)$, $\lim_{k\rightarrow \infty} \E{x(k)} $ converges to \eqref{e:markov_limit} asymptotically.
\end{theorem}

\begin{proof}

If $P$ is irreducible and aperiodic, and moreover there are finite states in the Markov chain of $\lambda(k)$, then $P$ has unique stationary distribution, denoted as $\pi^*$.

To simplify the presentation, denote $M=B-AW$. We first write the recursive equation for $q_i(k)$ and $Q_i(k)$ which are defined in \eqref{e:qiQi}. 
For \eqref{e:our_model_new}, 
\begin{align}
q_j(k+1) & = \E\big( x(k+1)\mathds{1}_{\{\lambda(k+1)=j\}} \big) \\
& = \sum_{i\in\N}\E\big( (Mx(k)+D\lambda(k)a) \nonumber\\
&\quad  \quad  \quad  \mathds{1}_{\{\lambda(k+1)=j\}}\mathds{1}_{\{\lambda(k)=i\}} \big) \\
& = \sum_{i\in\N} p_{ij} M q_i(k) + \sum_{i\in\N}p_{ij}\diag(i)a \pi_i(k),
\end{align}
and
\begin{align}
& Q_j(k+1) \nonumber \\
= & \E\big( x(k+1)x^\top(k+1)\mathds{1}_{\{\lambda(k+1)=j\}} \big) \\
= & \E\big( (Mx(k)+D\lambda(k)a)(Mx(k)+D\lambda(k)a)^\top \nonumber\\
 & \quad \quad \mathds{1}_{\{\lambda(k+1)=j\}}\mathds{1}_{\{\lambda(k)=i\}} \big)\\
= & \sum_{i\in\N} p_{ij}MQ_i(k)M^\top \nonumber \\
 & + \sum_{i\in\N} p_{ij}Mq_i(k)a^\top \diag(i) \nonumber \\
 & + \sum_{i\in\N} p_{ij}\diag(i)a q^\top_i(k)M^\top \nonumber\\
 & + \sum_{i\in\N} p_{ij} \diag(i)aa^\top \diag(i)\pi_i(k).
\end{align}
For any $V=(V_1,\ldots,V_N)$ where $V_i\in\R^{n\times n}$, we define the following operators
\begin{align}
\calT_j(V):= &  \sum_{i=1}^N p_{ij}MV_iM^\top 
\end{align}   
and $\calT : = (\calT_1,\ldots,\calT_N)$. Notice that we can write $\calT(V)$ as
\begin{align}
\vect(\calT(V)) = (P^\top\otimes I_N)\otimes (M\otimes M)\vect(V).
\end{align} 
Furthermore, we denote 
\begin{align}
R_j(q(k)):= & \sum_{i\in\N} p_{ij}\big(Mq_i(k)a^\top \diag(i) \nonumber \\
 &  +\diag(i)a q^\top_i(k)M^\top \nonumber\\
 &  +\diag(i)aa^\top \diag(i)\pi_i(k)\big) 
\end{align}
and $R(k,q):=(\ldots,R_i(k,q),\ldots)$ for all $i\in\N$, and 
\begin{align}
\calB(q(k)) & = (P^\top\otimes I_{2^n})\diag(M,\ldots,M)q(k) \\
& = (P^\top\otimes M) q(k) \\
\psi(k) & = \begin{bmatrix}
\sum_{i\in\N}p_{i1}\diag(i)a \pi_i(k) \\
\vdots \\
\sum_{i\in\N}p_{i,2^n}\diag(i)a \pi_i(k)
\end{bmatrix}.
\end{align}
Then the recursive equation of $q(k)$ and $Q(k)$ is given as 
\begin{align}
q(k+1) & = \calB(q(k))+\psi(k)\\
Q(k+1) & = \calT(Q(k))+R(q(k)).
\end{align}

Next we prove that $\rho(\calT)<1$ if and only if $\rho(M)<1$. Indeed, by Lemma 1 in \cite{Kubrusly1985}, we have that $\rho(\calT)<1$ if and only if 
\begin{align}
\lim_{k\rightarrow\infty}\|\calT^k(V)\|_1 = 0, \forall V=(V_1,\ldots,V_N) \textnormal{ with } V_i>0
\end{align}
which is equivalent to $\lim_{k\rightarrow\infty} \vect(\calT^k(V)) = 0$. By the fact that 
\begin{align}
\vect(\calT^k(V)) = \Big((P^\top\otimes I_N)\otimes (M\otimes M)\Big)^k\vect(V),
\end{align}
we have $\lim_{k\rightarrow\infty} \vect(\calT^k(V)) = 0$ for any $V$ is equivalent to
\begin{align}
&\rho((P^\top\otimes I_N)\otimes (M\otimes M)) \\
= & \rho(P^\top)\rho(M) \\
= & \rho(M) \\
< & 1. 
\end{align}
Furthermore, notice that if $\rho(M)< 1$, then $\rho(\calB)<1$. Hence, by Proposition 3.37 and 3.38 in \cite{Costa2005}, which shows system \eqref{e:our model} is mean square stable if and only if $\rho(\calT)<1$, we have the mean square stability is equivalent to $\rho(M)<1$. 

The remaining task is to show that the mean converges to \eqref{e:markov_limit}. 
Since we have $\lim_{k\rightarrow\infty}\sum_{\ell=0}^{k}(M)^\ell= (I-M)^{-1}$, then in order to prove $\mu(k)$ converge to \eqref{e:markov_limit}, it is equivalent to show that 
\begin{equation}
\lim_{k\rightarrow\infty} \mu(k)- \big(\sum_{\ell=0}^{k-1}(M)^\ell \big) (\pi^*\otimes I_{n})S=0.
\end{equation}
Indeed, 
\begin{align}
y(k):=& \mu(k)- \big(\sum_{\ell=0}^{k-1}(M)^\ell \big) (\pi^*\otimes I_{n})S \\
= & (M)^{k} \mu(0) \\
& + \sum_{\ell=0}^{k-1} (M)^{k-\ell} ((\pi(0)P^\ell)\otimes I_{n})S \\
& - \big(\sum_{\ell=0}^{k}(M)^\ell \big)(\pi^*\otimes I_{n})S \\
= & (M)^{k} \mu(0) \\
& + \big( \sum_{\ell=0}^{k-1} (M)^{k-\ell} \big) (((\pi(0)-\pi^*)P^\ell)\otimes I_{n})S.
\end{align} 
Then it can shown in a straightforward manner that $\|y(k)\|\rightarrow 0$ as $k\rightarrow \infty$, since $M$ is stable and $(\pi(0)-\pi^*)P^\ell\rightarrow 0$ as $\ell\rightarrow \infty$.
\end{proof}

\section{Numerical study}\label{s:simulation}

\begin{example}
In this example, we demonstrate the mean square stability of system \eqref{e:our model}. Here we consider the system \eqref{e:our model} with the matrices are set to be
\begin{align}\label{ex:mat_setting}
B&=\begin{bmatrix}
    0.2931 &   0.0660  &  0.0948  &  0.3384  &  0.2076 \\
    0.3670 &   0.3348  &  0.1038  &  0.1203  &  0.0741 \\
    0.0859 &   0.3849  &  0.0686  &  0.0489  &  0.4117 \\
    0.3084 &   0.1009  &  0.1948  &  0.3477  &  0.0482 \\
    0.3110 &   0.1164  &  0.1234  &  0.2924  &  0.1569 
\end{bmatrix}, \nonumber\\
W & = \begin{bmatrix}
    0.2846 &   0.2561  &  0.2067 &   0.0504  &  0.2021\\
    0.2423 &   0.3299  &  0.0407 &   0.0753  &  0.3118\\
    0.0955 &   0.3605  &  0.4363 &   0.0398  &  0.0680\\
    0.3572 &   0.3283  &  0.0169 &   0.2435  &  0.0541\\
    0.2922 &   0.2459  &  0.0044 &   0.1127  &  0.3449
\end{bmatrix}
\end{align}
and the learning rate $A=0.1I_5$. Let the random variable $\lambda(k)$ be i.i.d. with uncorrelated components, and set the probability of $\lambda_i(k)=0$ to 0.5 for all $i\in\calI$ and $k\geq 0$. Fig.~\ref{fig:1_onetraj} depicts one result for the system \eqref{e:our model}. It can be seen that since the randomness is added to every iteration of the system, the trajectories are oscillating. However, the expectation of the state $x(k)$, given as in Fig~\ref{fig:1_exp}, converge to $0.5\mathds{1}$ which also follows from Theorem \ref{thm:main2}.

\begin{figure}
\centering
\includegraphics[width=0.48\textwidth]{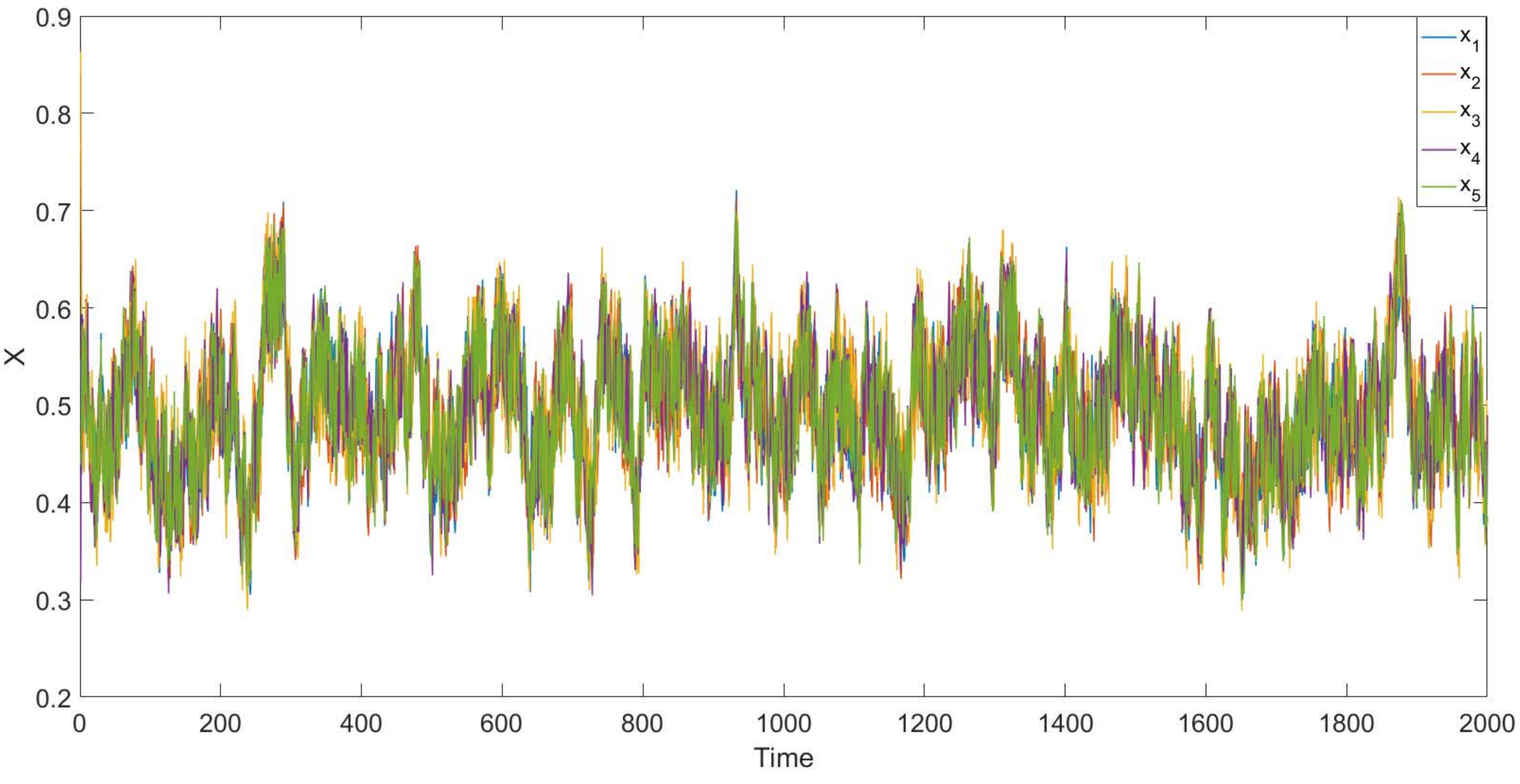}
\caption{The trajectories of the system \eqref{e:our model} with i.i.d. $ \lambda(k)$ and matrices as in \eqref{ex:mat_setting}.}\label{fig:1_onetraj}
\end{figure}

\begin{figure}
\centering
\includegraphics[width=0.48\textwidth]{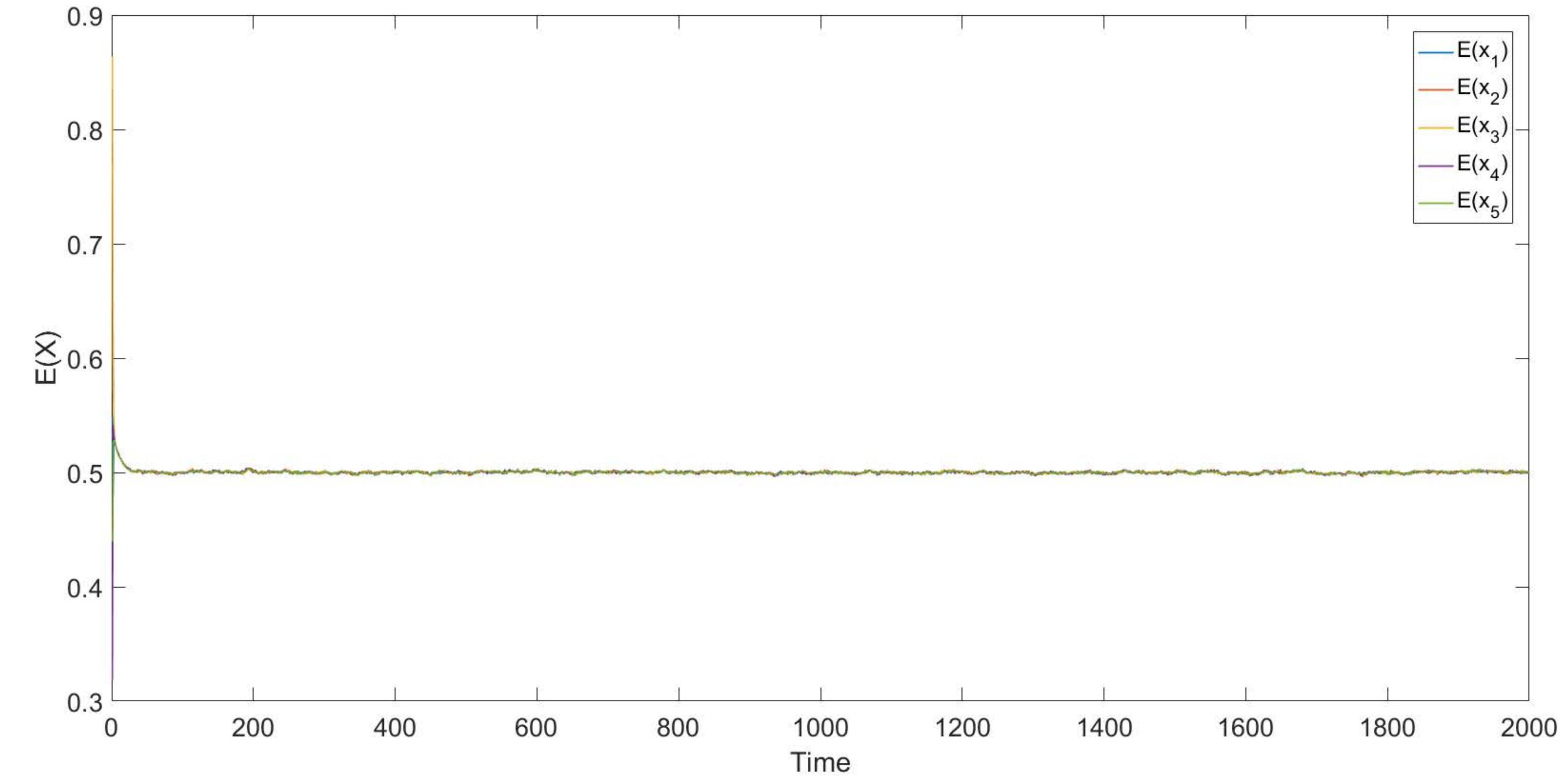}
\caption{The evolution of the expectation of the states of system \eqref{e:our model} with i.i.d. $ \lambda(k)$ and matrices as in \eqref{ex:mat_setting}.}\label{fig:1_exp}
\end{figure}
\end{example}

\begin{example}\label{ex: non-converging}
In this example, we provide one scenario which verifies the discussion in Remark \ref{re:non-converging}. Here, we consider the system \eqref{e:our model} with 
\begin{align}
B&=\begin{bmatrix}
    0.1612  &  0.1406  &  0.1096  &  0.2779  &  0.3108 \\
    0.0075  &  0.6307  &  0.1664  &  0.1910  &  0.0044  \\
    0.0219  &  0.3147  &  0.3403  &  0.2353  &  0.0878 \\
         0  &       0  &       0  &       0  &  1 \\
         0  &       0  &       0  &  1       &  0
\end{bmatrix}
\end{align}
and $W=B$, $A = \diag(0.5,0.5,0.5,0,0)^\top$. In this case we have $\rho(B-AW)=1$ which violate the assumption in Theorem \ref{thm:main2}.  Fig.~\ref{fig:2_onetraj} depicts one trajectory for this system and Fig.~\ref{fig:2_exp} shows the evolution of expectation of $x(k)$.  It can be seen that the expectation will not converge to a constant vector in this case.

\begin{figure}
\centering
\includegraphics[width=0.48\textwidth]{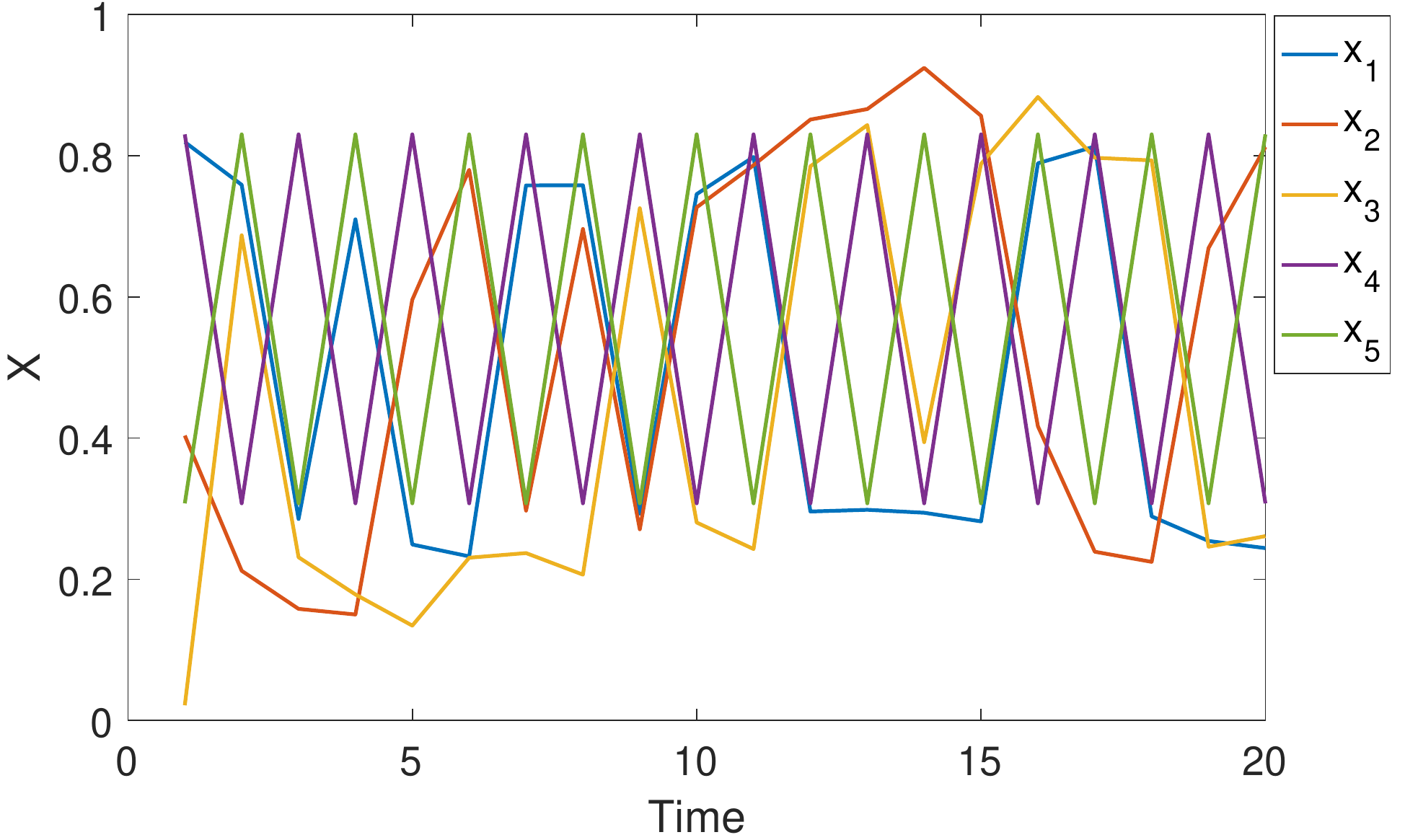}
\caption{The trajectories of the system \eqref{e:our model} with i.i.d. $ \lambda(k)$ and matrices as in \eqref{ex:mat_setting}.}\label{fig:2_onetraj}
\end{figure}

\begin{figure}
\centering
\includegraphics[width=0.48\textwidth]{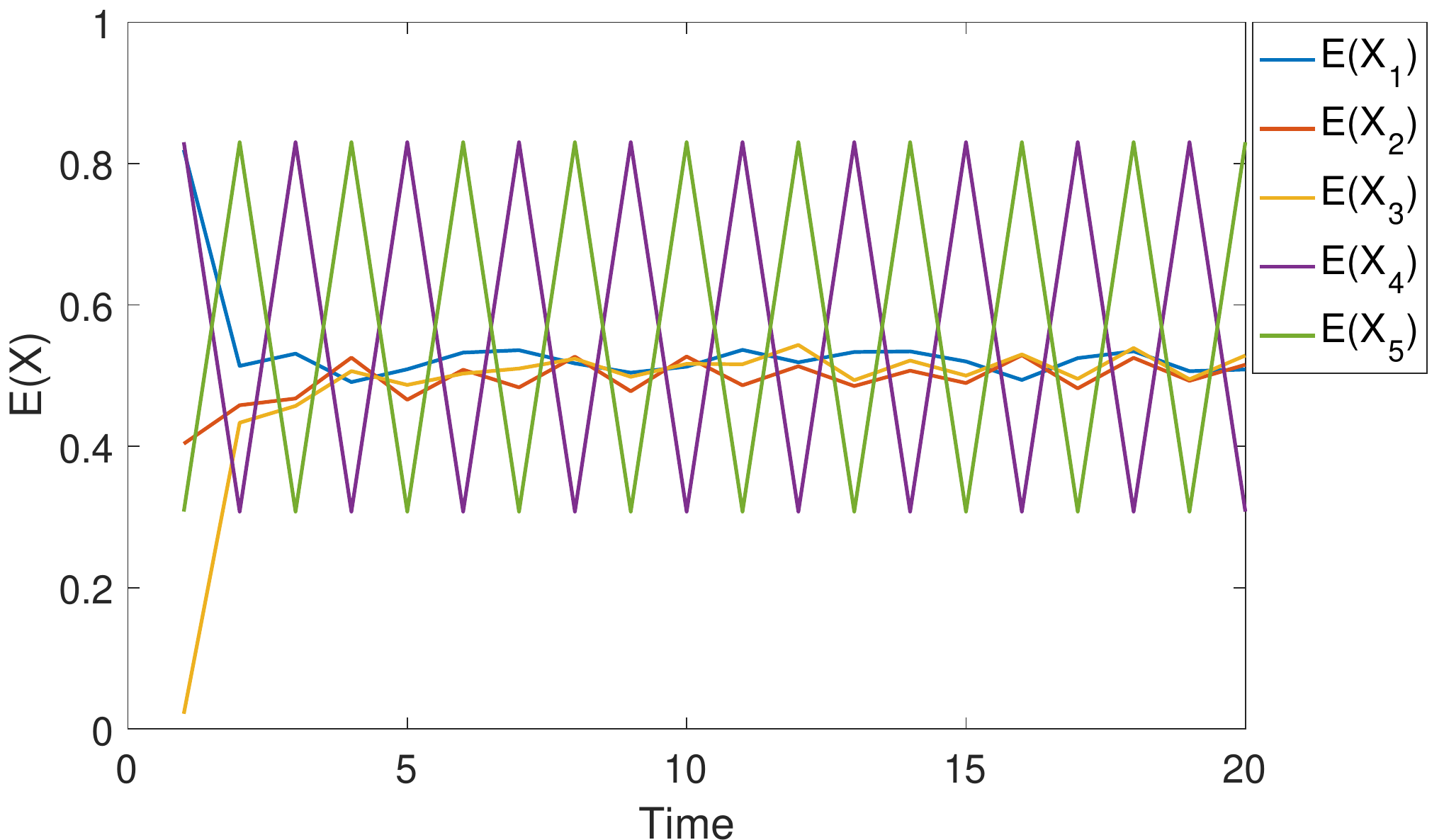}
\caption{The evolution of the expectation of the states of system \eqref{e:our model} with i.i.d. $ \lambda(k)$ and matrices as in \eqref{ex:mat_setting}.}\label{fig:2_exp}
\end{figure}
\end{example}

\begin{example}
In this example, we consider the system \eqref{e:our model} with the same matrices as in \eqref{ex:mat_setting}, but the vector $\lambda(k)$ is a Markov chain as in subsection \ref{ss:markov}. Here we consider the case that all the components of $\lambda$ are independent, and for each component the distribution evolve according to the same transition matrix. It can be seen that for this type of random vector $\lambda(k)$, the expectation of $\lambda(k)$ converge to $\spa\{\mathds{1}\}$.  Moreover, the term $\sum_{i\in\N} \pi_i \diag(i)a$ in \eqref{e:markov_limit} belongs to $\spa\{\mathds{1} \}$. Then, as can be seen from Fig.~\ref{fig:3_exp}, the expectation of $x(k)$ converge to consensus.  

\begin{figure}
\centering
\includegraphics[width=0.48\textwidth]{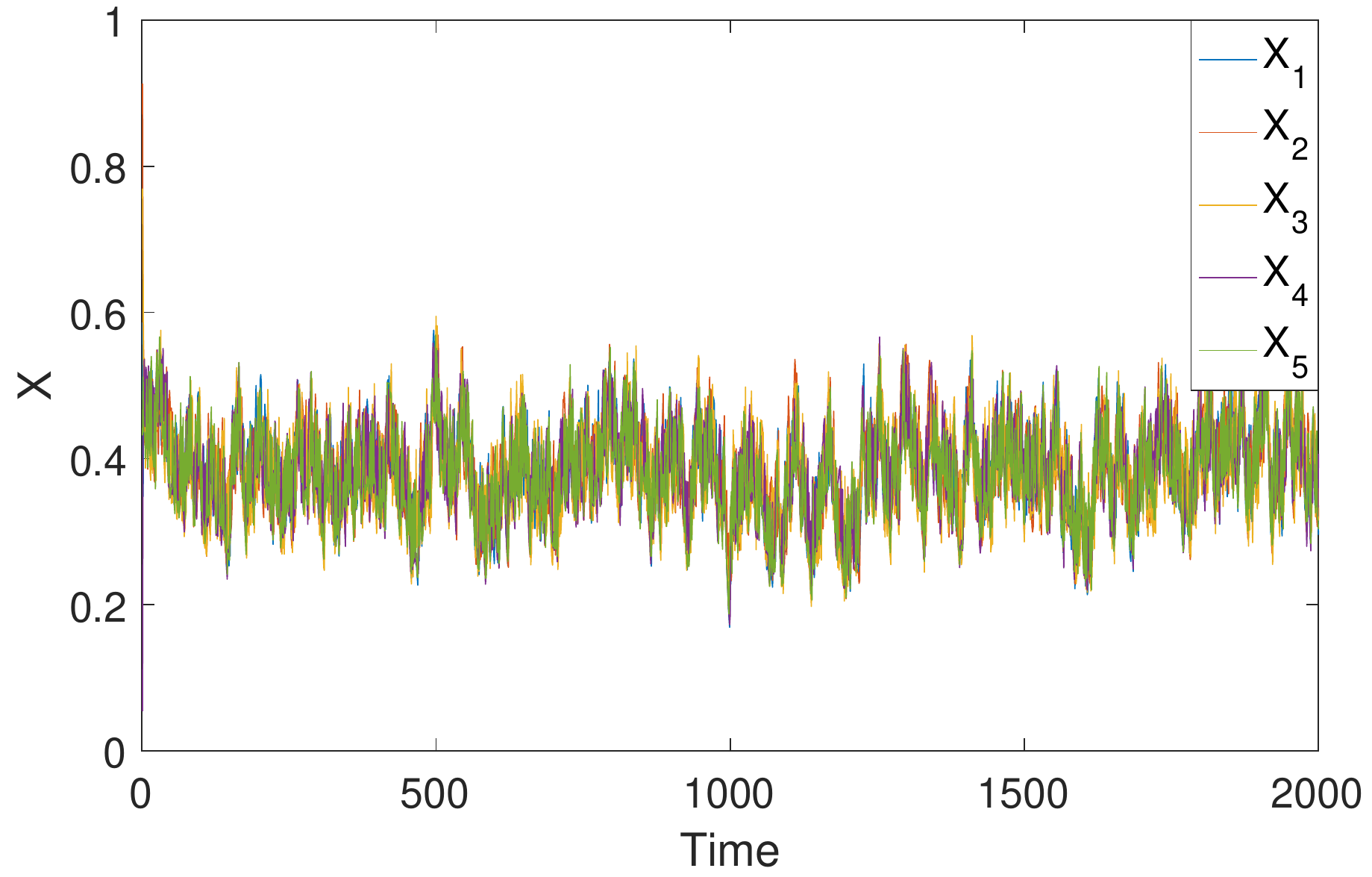}
\caption{The trajectories of the system \eqref{e:our model} with Markovian $ \lambda(k)$ and matrices as in \eqref{ex:mat_setting}.}\label{fig:3_onetraj}
\end{figure}

\begin{figure}
\centering
\includegraphics[width=0.48\textwidth]{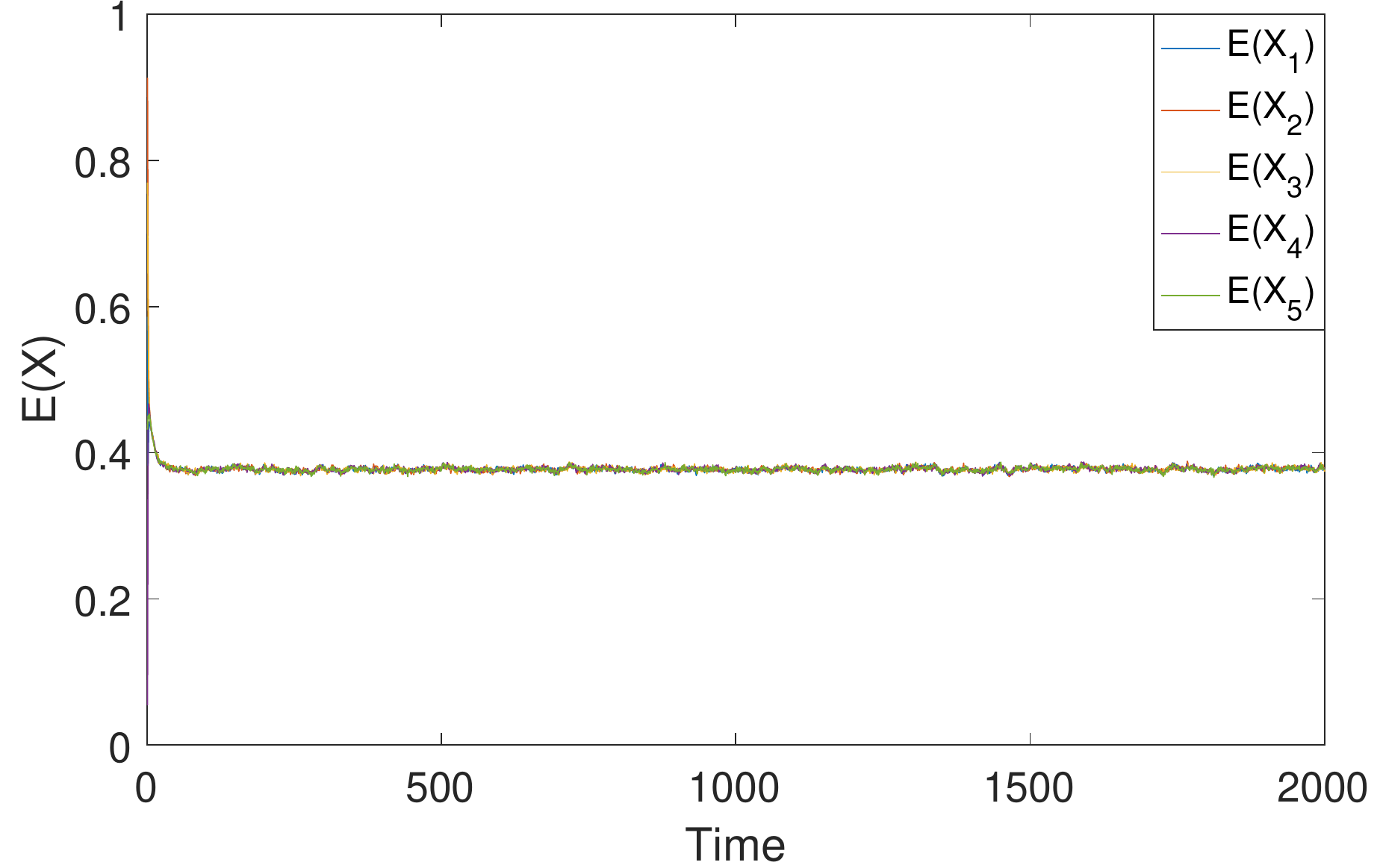}
\caption{The evolution of the expectation of the states of system \eqref{e:our model} with Markovian $ \lambda(k)$ and matrices as in \eqref{ex:mat_setting}.}\label{fig:3_exp}
\end{figure}
\end{example}

\section{Conclusion}\label{s:conclusion}

In this paper, we studied the modeling of the human behavior in social network along the path proposed by the author in \cite{epstein2014agent_zero}. We have focused on the irrational component of human cognitive process and proposed one general model. This model contains the well-known Rescorla-Wagner and Friedkin-Johnsen model as special cases. The sufficient and necessary condition is provided for the mean square stability of our system. For the future directions, we will mainly focus on how does the incorporation of our model into the human cognitive process will affect the human behavior for large networks.



\bibliographystyle{plain} 
\bibliography{ref}

\end{document}